\documentclass[a4paper,11pt]{amsart}
\title[A free boundary problem of a nonlinear diffusion equation]
{Refined estimates for the propagation speed of the transition solution to a free boundary problem with a nonlinearity of combustion type}
\author[C. Lei, H. Matsuzawa, R. Peng and M. Zhou
]{
Chengxia Lei${}^\dag$, Hiroshi Matsuzawa${}^\S$, Rui Peng${}^\ddag$ and Maolin Zhou${}^\natural$}

\thanks{${}^\dag$School of Mathematics and Statistics, Jiangsu Normal University, Xuzhou, 221116, Jiangsu Province, China. (Email: {\tt leichengxia001@163.com})}
\thanks{${}^\S$National Institute of Technology, Numazu College, 3600 Ooka, Numazu City, Shizuoka 410-8501, Japan. (Email: {\tt hmatsu@numazu-ct.ac.jp})}
\thanks{${}^\ddag$School of Mathematics and Statistics, Jiangsu Normal University, Xuzhou, 221116, Jiangsu Province, China. (Email: {\tt pengrui\,$\b{}$\,seu@163.com})}
\thanks{${}^\natural$School of Science and Technology, University of New England, Armidale, NSW, 2351, Australia. (Email: {\tt mzhou6@une.edu.au})}
\thanks{H. Matsuzawa was partially supported by Grant-in-Aid for Scientific Research (C) 17K05340, and R. Peng was partially supported by NSF of China (No. 11671175, 11571200), the Priority Academic Program Development of Jiangsu Higher Education Institutions,
Top-notch Academic Programs Project of Jiangsu Higher Education Institutions (No. PPZY2015A013) and Qing Lan Project of Jiangsu Province.}
\date{\today}
\usepackage{amsmath, enumerate, color, mathrsfs, cases}

\makeatletter

\@addtoreset{equation}{section}
\makeatother
\theoremstyle{definition}

\theoremstyle{plain}
\newtheorem{thm}{Theorem}[section]
\newtheorem{theorem}{Theorem}
\newtheorem{proposition}{Proposition}[section]

\newtheorem{lemma}[proposition]{Lemma}

\theoremstyle{definition}
\newtheorem{remark}{Remark}

\begin{document}

\begin{abstract}
We are concerned with the nonlinear problem $u_t=u_{xx}+f(u)$, where $f$ is of combustion type,
coupled with the Stefan-type free boundary $h(t)$. According to \cite{DL,DLZ},
for some critical initial data, the transition solution $u$ locally uniformly converges to $\theta$,
which is the ignition temperature of $f$, and the free boundary satisfies  $h(t)=C\sqrt{t}+o(1)\sqrt{t}$
for some positive constant $C$ and all large time $t$.
In this paper, making use of two different approaches, we establish more accurate upper and lower bound estimates on
$h(t)$ for the transition solution, which suggest that the nonlinearity $f$ can essentially influence the propagation speed.
\end{abstract}

\maketitle

\section{Introduction and Main Results}

In this paper, we consider the following free boundary problem of nonlinear diffusion equations:
\begin{align}\label{problem}
\left\{
\begin{array}{ll}
u_t=u_{xx}+f(u),& t>0,\ g(t)<x<h(t), \\
u(t,g(t))=u(t,h(t))=0, & t>0, \\
g'(t)=-\mu u_x(t, g(t)), & t>0, \\
h'(t)=-\mu u_x(t, h(t)), & t>0, \\
-g(0)=h(0)=h_0, u(0,x)=u_0(x), & -h_0\le x\le h_0,
\end{array}\right.
\end{align}
where $x=g(t)$ and $x=h(t)$ are the expanding fronts, $\mu$, $h_0$ are given positive constants.
The nonlinear reaction term $f$ is a locally Lipschitz continuous function satisfying
\begin{align}\label{f1}
f(0)=0,\ \ \ f(u)<0\ \ {\rm for}\ u>1.
\end{align}
Throughout the paper, unless otherwise specified, we assume that $f$ is of combustion type:
\begin{align}\label{combustion}
\begin{split}
&f(u)=0\ \ {\rm in}\ [0,\theta],\ \ \ f(u)>0\ \ {\rm in}\ (\theta, 1), \\
&f(u)<0\ \ {\rm in}\ (1,\infty)
\end{split}
\end{align}
for some constant $\theta\in(0,1)$.

For any given $h_0>0$, the initial function $u_0(x)$ satisfies
\begin{align}\label{u0}
u_0\in\mathscr{X}(h_0):=\left\{\phi\in C^2[-h_0, h_0]:\begin{tabular}{ll}
$\phi(-h_0)=\phi(h_0)=0,\ \phi'(-h_0)>0,$ \\
$\phi'(h_0)<0,\ \phi(x)>0\ {\rm in}\ (-h_0,h_0)$. \end{tabular}\right\}
\end{align}
It was shown by \cite{DL} that under condition \eqref{f1}, \eqref{problem} has a unique globally defined classical solution
$(u(t,x), h(t), g(t))$. In addition, $g'(t)<0$ and $h'(t)>0$ for $t>0$ and therefore
\begin{align*}
g_{\infty}:=\lim_{t\to\infty}g(t)\in [-\infty, -h_0),\ \ h_{\infty}:=\lim_{t\to\infty}h(t)\in(h_0, +\infty]
\end{align*}
always exist.

Denote $\mathbb{R}=(-\infty,\infty)$. In \cite{DL}, the authors obtained the following trichotomy result
and the sharp threshold dynamics when $f$ satisfies \eqref{combustion}.

\begin{theorem}[\mbox{\cite[Theorem 1.4]{DL}}]\label{theoremDL}
Assume that $f$ is of combustion type, and $h_0>0$, $u_0\in\mathscr{X}(h_0)$. Then one of the following situations occurs:
\begin{enumerate}[\rm (i)]
\item {\bf Spreading:} $(g_{\infty}, h_{\infty})=\mathbb{R}$ and
\begin{align*}
\lim_{t\to\infty}u(t,x)=1\ \ \mbox{\it locally\ uniformly\ in}\ \mathbb{R};
\end{align*}
\item {\bf Vanishing:} $(g_{\infty}, h_{\infty})$ is a finite interval and
\begin{align*}
\lim_{t\to\infty}\max_{x\in[g(t), h(t)]}u(t,x)=0;
\end{align*}
\item {\bf Transition:} $(g_{\infty}, h_{\infty})=\mathbb{R}$ and
\begin{align*}
\lim_{t\to\infty}u(t,x)=\theta\ \ \mbox{\it locally\ uniformly\ in}\ \mathbb{R}.
\end{align*}
\end{enumerate}
Moreover, if $u_0=\nu\phi$ ($\nu>0$) for some $\phi\in\mathscr{X}(h_0)$, then there exists $\nu^*=\nu^*(h_0,\phi)\in(0, \infty]$ such that
 vanishing happens when $0<\nu<\nu^*$, spreading happens when $\nu>\nu^*$, and transition happens when $\nu=\nu^*$.
\end{theorem}

\begin{remark}In \cite{DL}, it was assumed that $f$ is $C^1$; in view of the argument of \cite{DM}, one can easily see that $f$
can be relaxed to be locally Lipschitz and Theorem \ref{theoremDL} remains true.
\end{remark}

%
In the paper \cite{DMZ}, the authors derived the following propagation speed when spreading happens.

\begin{theorem}[\mbox{\cite[Theorem 1.2]{DMZ}}]
Assume that $f$ is of combustion type with $f\in C^1$ and $f'(1)<0$. If spreading happens, then there exist $H_{\infty}$, $G_{\infty}\in\mathbb{R}$ such that
\begin{align*}
&\lim_{t\to\infty}(h(t)-c^*t-H_{\infty})=0,\ \ \lim_{t\to\infty}h'(t)=c^*, \\
&\lim_{t\to\infty}(g(t)+c^*t-G_{\infty})=0,\ \ \lim_{t\to\infty}g'(t)=-c^*,
\end{align*}
where $c^*$ is uniquely determined by the following equation
\begin{align*}
\left\{\begin{array}{l}
q_{zz}-c^*q_z+f(q)=0\ \ {\rm for}\ z\in(0,\infty), \\
q(0)=0,\ \mu q(0)=c^*,\ q(\infty)=1,\ q(z)>0\ \ {\rm for}\ z\in(0,\infty).
\end{array}\right.
\end{align*}
\end{theorem}

The more recent paper \cite{DLZ} established some estimates on $h(t)$ in the transition case.

\begin{theorem}[\mbox{\cite[Theorem 1.2]{DLZ}}]\label{DLZthm12}
Assume that $f$ is of combustion type and transition happens. Then we have
\begin{align*}
h(t), -g(t)=2\xi_0\,\sqrt{t}\{1+o(1)\},\ \ \ {\rm as}\ t\to\infty,
\end{align*}
where $\xi_0>0$ is uniquely determined by
\begin{align}\label{xi0}
2\xi_0e^{\xi_0^2}\int_0^{\xi_0}e^{-s^2}ds=\mu\theta.
\end{align}
\end{theorem}

Regarding related studies on the free boundary problem, one may refer to \cite{DG,DLI,DDL,DMZ2} and the references therein.

The purpose of the current paper is to obtain more delicate estimates on the small order $o(1)\sqrt{t}$ in Theorem \ref{DLZthm12}
for a special type of nonlinearity $f$. Precisely, by assuming that
there exist constants $p\ge 1$ and small $\sigma>0$ such that
\begin{align}\label{f2}
f(u)=(u-\theta)^p\ \ {\rm in}\ [\theta, \theta+\sigma),
\end{align}
we aim to investigate how the power $p$ in \eqref{f2} influences the asymptotic behavior $h(t)$ for large time $t$ in the transition case.
Our main result reads as follows.

\begin{thm}\label{theoremA}
Assume that $f$ is of combustion type satisfying \eqref{f2} for some $p\ge 1$ and $\sigma>0$, and that the transition happens. Then the following assertions hold.
\begin{enumerate}[{\rm (1)}]
\item If $p>3$, we have
\begin{align*}
2\xi_0\sqrt{t}+m t^{\frac{1}{2}-\frac{1}{p-1}}\le -g(t),\ \ h(t)\le 2\xi_0\sqrt{t}+M t^{\frac{1}{2}-\frac{1}{p+1}},\ \ \mbox{as $t\to\infty$},
\end{align*}
where $m$ and $M$ are some positive constants and $\xi_0>0$ is the constant uniquely given by \eqref{xi0}.
\item If $1\leq p\leq 3$, we have
\begin{align*}
2\xi_0\sqrt{t}+O(1)\le -g(t),\ \ h(t)\le 2\xi_0\sqrt{t}+M t^{\frac{1}{2}-\frac{1}{p+1}},\ \ \mbox{as $t\to\infty$},
\end{align*}
where $\xi_0>0$ is the constant uniquely given by \eqref{xi0}.
\end{enumerate}
\end{thm}

\begin{remark}
Theorem \ref{theoremA} shows that when $p>3$, $h(t)-2\xi_0\sqrt{t}$ and $g(t)+2\xi_0\sqrt{t}$ are unbounded for large time $t$
while $h(t)-2\xi_0\sqrt{t}$ and $g(t)+2\xi_0\sqrt{t}$ become bounded when $p=1$. Therefore,
our results reveal that the nonlinear reaction term $f$ has a qualitative impact on the propagation speed for the transition solution.
\end{remark}

\begin{remark}
In light of Theorem \ref{theoremA}, it is unclear to us that $h(t)-2\xi_0\sqrt{t}$ and $g(t)+2\xi_0\sqrt{t}$ are unbounded or bounded for large time $t$ when $1<p\leq 3$.
It would be of interest to investigate the refined propagation speed of the transition solution for the nonlinearity $f$ satisfying \eqref{f2} with $0<p<1$ or $1<p\leq 3$.
\end{remark}

\section{Preliminary results}

In this section, we give some basic facts, which will be used frequently in the forthcoming sections.
We shall begin with some comparison principles associated with problem \eqref{problem}. The first one, which
will be used later to estimate the lower bound of $h(t)$, is due to \cite{DL}.

\begin{lemma}[\mbox{\cite[Lemma 2.2]{DL}}]\label{comp1}
Suppose that $f\in C^{0,1}_{\rm loc}$ and $f(0)=0$, $T\in (0,\infty)$, $\xi$, $\overline{h}\in C^1[0,T]$, $\overline{u}\in C(\overline{D}_T)\cap C^{1,2}(D_T)$ with
$D_T=\{(t,x)\in\mathbb{R}^2:0<t\le T,\ \xi(t)<x<\overline{h}(t)\}$, and
\begin{align*}
\left\{
\begin{array}{ll}
\overline{u}_t\ge \overline{u}_{xx}+ f(\overline{u}),& 0<t\le T,\ \xi(t)<x<\overline{h}(t), \\
\overline{u}(t,\xi(t))\ge u(t, \xi(t)),& 0<t\le T,  \\
\overline{u}(t,\overline{h}(t))=0,\ \overline{h}^{\prime}(t)\ge -\mu\overline{u}_x(t, \overline{h}(t)),& 0<t\le T. \\
\end{array}\right.
\end{align*}
If
\begin{align*}
\xi(t)\ge g(t)\ {\it in}\ [0,T],\ h_0\le \overline{h}(0),\ u_0(x)\le\overline{u}(0,x)\ \ {\it in}\ [\xi(0),h_0],
\end{align*}
where $(u,g,h)$ is a solution to {\rm (\ref{problem})}, then
\begin{align*}
&h(t)\le \overline{h}(t)\ {\it in}\ (0,T],\\
&u(t,x)\le\overline{u}(t,x)\ \ {\it for}\ t\in (0,T]\ \ {\it and}\ \ \xi(t)<x<\overline{h}(t).
\end{align*}
\end{lemma}

The function $\overline{u}$ or the triple $(\overline{u}, \xi, \overline{h})$ in Lemma \ref{comp1} is usually called an
upper solution of problem \eqref{problem}. We can define a lower solution by reversing all the inequalities in the suitable places above.

The following version of comparison principle will be used to estimate the upper bound of $h(t)$.

\begin{lemma}\label{comp2}
Suppose that $f\in C^{0,1}_{\rm loc}$ and $f(0)=0$, $T\in (0,\infty)$, $\xi$, $\overline{h}\in C^1[0,T]$, $\overline{u}\in C(\overline{D}_T)\cap C^{1,2}(D_T)$ with
$D_T=\{(t,x)\in\mathbb{R}^2:0<t\le T,\ \xi(t)<x<\overline{h}(t)\}$, and
\begin{align*}
\left\{
\begin{array}{ll}
\overline{u}_t\ge \overline{u}_{xx}+ f(\overline{u}),& 0<t\le T,\ \xi(t)<x<\overline{h}(t), \\
\left\{
\begin{array}{ll}
\overline{u}(t, \xi(t))\ge 0,& {\rm if}\ h(t)\le \xi(t), \\
\overline{u}(t, \xi(t))\ge u(t, \xi(t)), & {\rm if}\ \xi(t)<h(t),
\end{array}\right. & 0<t\le T, \\
\overline{u}(t, \overline{h}(t))=0, & 0<t\le T, \\
\overline{h}^{\prime}(t)\ge -\mu\overline{u}_x(t, \overline{h}(t)),& 0<t\le T.
\end{array}\right.
\end{align*}
If
\begin{align*}
&\xi(t)\ge 0,\ \ 0\le t\le T, \\
&h_0\le \xi(0)<\overline{h}(0), \\
&\overline{u}(0, x)\ge 0,\ \ \xi(0)\le x\le \overline{h}(0), \\
&\overline{u}(0, x)>0,\ \ \xi(0)<x<\overline{h}(0),
\end{align*}
and $u$ is the unique solution to problem \eqref{problem}, then
\begin{align*}
& u(t, x)\le \overline{u}(t, x),\ \ 0<t\le T, \ x\in (0, h(t))\cap (\xi(t), \overline{h}(t)), \\
& h(t)<\overline{h}(t),\ \ 0<t\le T.
\end{align*}
\end{lemma}
\begin{proof}
If $h(t)\le \xi(t)$ for $t\in(0, T]$, nothing is needed to prove. Suppose that there exist $t_*$, $t^*\in (0,T]$ with $t_*<t^*$ such that
\begin{align*}
&h(t_*)=\xi(t_*), \\
&\xi(t)<h(t)<\overline{h}(t),\ \ t\in (t_*, t^*), \\
&h(t^*)=\overline{h}(t^*).
\end{align*}
Then we have
\begin{align}\label{hprime}
h'(t^*)\ge \overline{h}'(t^*).
\end{align}

Let $\Omega_*=\{(t,x)\in\mathbb{R}^2| t_*<t\le t^*,\ \xi(t)<x<h(t)\}$. It follows from the maximum principle for parabolic equations that
\begin{align*}
u(t,x)\le \overline{u}(t,x)\ \ {\rm in}\ \Omega_*.
\end{align*}
Since $\overline{u}(t, h(t))>0=u(t, h(t))$ for $t\in (t_*, t^*)$, the strong maximum principle for parabolic equations infers that
\begin{align*}
u(t,x)<\overline{u}(t,x)\ \ {\rm in}\ \Omega_*.
\end{align*}
By setting $w(t,x)=\overline{u}(t, x)-u(t,x)$, then $w$ satisfies
\begin{align*}
w_t-w_{xx}-c(t,x)w\ge 0\ \ {\rm in}\ \Omega_*
\end{align*}
for some bounded function $c(t,x)$. We also have
\begin{align*}
&w(t,x)>0\ \ \ {\rm in}\ \Omega_*, \\
&w(t^*, h(t^*))=\overline{u}(t^*, \overline{h}(t^*))-u(t^*, h(t^*))=0.
\end{align*}
An application of the Hopf lemma allows one to conclude that $w_x(t^*, h(t^*))<0$, that is,
\begin{align*}
\overline{h}'(t^*)\ge -\mu\overline{u}_x(t^*, \overline{h}(t^*))>-\mu u_x(t^*, h(t^*))=h'(t^*),
\end{align*}
which is a contradiction against \eqref{hprime}.
\end{proof}

The following lemma also plays a key role in our later analysis. To stress the dependence of the unique solution $(u(t,x), g(t), h(t))$ of \eqref{problem}
on the initial function $u_0$, we will sometimes use notation $u(t,x)=u(t,x; u_0)$, $h(t)=h(t; u_0)$ and $g(t)=g(t; u_0)$.

\begin{lemma}[\mbox{\cite[Lemma 4.1]{DLZ}}]\label{DLZLa41}
Suppose that $f$ is of combustion type and that for any compactly supported functions $u_0^1$ and $u_0^2$ satisfying \eqref{u0} such that
the transition happens for $u(t, x; u_0^i)$, $i=1,2$. Then there exists a constant $C>0$ such that
\begin{align*}
|h(t;u_0^1)-h(t; u_0^2)|\le C,\ \ |g(t; u_0^1)-g(t; u_0^2)|\le C.
\end{align*}
\end{lemma}

In light of the above lemma, to prove Theorem \ref{theoremA}, it suffices to deal with some special initial function.
Without loss of generality, we fix an initial datum $\tilde{u}_0$ to be symmetrically decreasing and
\begin{equation}\label{u02}
\tilde{u}_0(x)=\tilde{u}_0(-x)\ \ {\rm and}\ \ \tilde{u}_0'(x)<0,\ \ x\in(0, h_0],
\end{equation}
so that the transition case happens. Consequently, the solution $(u,g,h)$ with such initial datum  $\tilde{u}_0$ satisfies
\begin{align}\label{symm}
\begin{split}
&g(t)=-h(t)\ \ {\rm for}\  t>0, \\
&u(t,x)=u(t, -x)\ \ {\rm for} \ t>0,\  x\in [-h(t), h(t)], \\
&u_x(t,x)<0\ \ \ {\rm for}\ t\ge 0\ {\rm and} \ x\in (0, h(t)].
\end{split}
\end{align}
\begin{remark} We note that the above chosen initial datum $\tilde{u}_0$ may vary with respect to the nonlinearity $f$.
Thus in what follows whenever such $\tilde{u}_0$ is used, it should be understood that the function $f$ is fixed first.
\end{remark}

With the above chosen initial datum $\tilde{u}_0$, we now need to show that when $p>3$, it holds
\begin{align*}
2\xi_0\sqrt{t}+m t^{\frac{1}{2}-\frac{1}{p-1}}\le h(t)\le 2\xi_0\sqrt{t}+M t^{\frac{1}{2}-\frac{1}{p+1}},\ \ \mbox{as\ $t\to\infty$}
\end{align*}
for some positive constants $m$ and $M$ and when $1\le p\le 3$, it holds
\begin{align*}
2\xi_0\sqrt{t}+O(1)\le h(t)\le 2\xi_0\sqrt{t}+M t^{\frac{1}{2}-\frac{1}{p+1}} ,\ \ \mbox{as\ $t\to\infty$}.
\end{align*}




\section{Proof of the main theorem}

Throughout this section, we assume that $f$ is of combustion type and that \eqref{f2} holds for $p\ge 1$
and $\sigma>0$, and that $(u,g,h)$ is the solution of \eqref{problem} with the initial datum $\tilde{u}_0$.
\subsection{Upper bound estimate}

We first observe that $u(t,0)>\theta$ for all $t\ge 0$. Otherwise there exists $t_0>0$ such that $u(t_0,0)\le\theta$ and hence $u(t_0, x)<\theta$ for
$x\in [-h(t_0), h(t_0)]\backslash\{0\}$. By the strong maximum principle we easily deduce $u(t, x)<\theta$ for $t>t_0$ and $x\in [-h(t), h(t)]$, which in turn implies
$u(t,x)\to 0$ as $t\to\infty$, contradicting the assumption that the transition happens. Combining this and \eqref{symm}, we can see that for each $t\ge 0$, there exists a unique
$\theta(t)\in (0, h(t))$ such that
\begin{align}\label{thetat}
u(t, \theta(t))=\theta.
\end{align}

To obtain an upper bound estimate of $h(t)$, as a first step, we will derive an estimate for $\theta(t)$. For this purpose, given a constant $b\in (0, 1-\theta)$,
let us recall the property of the initial value problem:
\begin{align}\label{ivp}
V''+f(V)=0,\ V(0)=\theta+b,\ V'(0)=0.
\end{align}

Through a phase plane analysis, \eqref{ivp} has a unique solution $V_b$ and numbers $l(b)$, $L(b)\in (0, \infty)$ with $l(b)<L(b)$ such that
\begin{align}\label{defoflb}
\begin{split}
&V_b(L(b))=0,\ \ V_b(l(b))=\theta, \\
&V_b'(x)<0\ \ {\rm for}\ x\in(0, L(b)].
\end{split}
\end{align}
Furthermore, we have
\begin{lemma}\label{lb} The following assertions hold.
\begin{enumerate}[{\rm (i)}]
\item If $p>1$, then $l(b)$ is strictly decreasing in $b\in(0, \sigma)$.
\item If $p=1$, then $l(b)=\frac{\pi}{2}$ for $b\in(0,\sigma)$.
\item If $p\geq1$, then $L(b)$ is strictly decreasing in $b\in(0,\sigma)$.
\end{enumerate}
\end{lemma}
\begin{proof}
By virtue of the equation for $V_b$, direct calculation gives
\begin{align*}
&-V_b''V_b'=f(V_b)V_b', \\
&\left(-\frac{(V_b')^2}{2}\right)'=(F(V_b))',
\end{align*}
where
 $$
 F(u)=\int_0^u f(s)ds.
 $$
By integrating we then obtain
 $$
 -\frac{V_b'(x)^2}{2}=F(V_b(x))-F(V_b(0)),
 $$
and so
 $$
 V_b'(x)^2=2[F(b+\theta)-F(V_b(x))],
 $$
that is,
 \begin{align}
 \frac{V_b'(x)}{\sqrt{2[F(b+\theta)-F(V_b(x))]}}=-1. \label{Vb0}
 \end{align}
Integrating the above identity over $[0, l(b)]$, we have
\begin{align*}
&\int_0^{l(b)}\frac{V_b'(x)}{\sqrt{2[F(b+\theta)-F(V_b(x))]}}dx=-l(b),\\
&\int_{b+\theta}^{\theta}\frac{ds}{\sqrt{2[F(b+\theta)-F(s)]}}=-l(b), \\
&\int_0^b\frac{ds}{\sqrt{G(b)-G(s)}}=l(b),
\end{align*}
where
 $$
 G(u)=2\int_0^u f(\theta+s)ds.
 $$

For $0<u<\sigma$ and $p>1$, it follows that
\begin{align}\label{Gu}
G(u)=2\int_0^u f(\theta+s)ds=2\int_0^u s^pds=\frac{2}{p+1}u^{p+1}
\end{align}
and
\begin{align}
l(b)&=\int_0^b\frac{ds}{\sqrt{\frac{2}{p+1}(b^{p+1}-s^{p+1})}}=\frac{\sqrt{p+1}}{\sqrt{2}b^{\frac{p+1}{2}}}\int_0^b \frac{ds}{\sqrt{1-\left(\frac{s}{b}\right)^{p+1}}} \nonumber \\
    &=\frac{\sqrt{p+1}}{\sqrt{2}b^{\frac{p+1}{2}}}b\int_0^1\frac{dx}{\sqrt{1-x^{p+1}}}=C_{p}b^{-\frac{p-1}{2}}, \label{lbest1}
\end{align}
where
 $$
 C_{p}=\sqrt{\frac{p+1}{2}}\int_0^1\frac{dx}{\sqrt{1-x^{p+1}}}.
 $$
This implies that $l(b)$ is strictly decreasing in $b\in (0, \sigma)$ if $p>1$.

When $p=1$, $l(b)=\frac{\pi}{2}$ is an obvious fact.

We next investigate $L(b)$. As $V_b(x)$ is a linear function over $[l(b), L(b)]$, we have
\begin{align*}
-V_b'(l(b))=\frac{\theta}{L(b)-l(b)}.
\end{align*}
This gives
\begin{align}\label{Lb2}
L(b)=l(b)-\frac{\theta}{V_b'(l(b))}.
\end{align}
By \eqref{Vb0} we know $V_b'(l(b))=-\sqrt{G(b)}$. Thus, from \eqref{Lb2} and \eqref{Gu} it follows
\begin{align}\label{Lb3}
L(b)=C_{p}b^{-\frac{p-1}{2}}+\sqrt{\frac{p+1}{2}}\theta b^{-\frac{p+1}{2}}\ \ \mbox{if}\;p>1,
\end{align}
and
$$L(b)=\frac{\pi}{2}+\theta b^{-1}\ \ \mbox{if}\;p=1.$$
Hence, $L(b)$ is strictly decreasing in $b\in (0, \sigma)$.
\end{proof}

\begin{lemma}[\mbox{\cite[Lemma 4.3]{DLZ}}]\label{DLZLa43}
If $0<b<\frac{1-\theta}{2}$, then $L(b)\to +\infty$ if and only if $b\to0$.
\end{lemma}

As one will see below, for any fixed $t\ge 1$, the sign-changing pattern of the function $u(t,x)-V_b(x)$ becomes crucial in obtaining a refined estimate for $\theta(t)$.
Indeed, such kind of properties were studied in Lemmas 4.6--4.8 in \cite{DLZ}.

By Lemma 4.5 of \cite{DLZ}, there exists $\delta_0>0$ such that for each $b\in(0,\delta_0)$,
 $$w_b(t,x):=u(t, x)-V_b(x)$$ satisfies $$w_b(1,0)>0>w_b(1, h(1))$$ and has a
 unique zero in $[0, h(1)]$, and the zero is nondegenerate.

As $t\mapsto w_b(t,x)$, $t\mapsto (w_b)_x(t,x)$, and $t\mapsto h(t)$ are all continuous uniformly in $x$, we see that
 for each $b\in(0, \delta_0)$ there exists $\epsilon_0>0$ small such that for each fixed $t\in[1-\epsilon_0, 1+\epsilon_0]$, $w_b(t,x)$ satisfies
  $$w_b(t, 0)>0>w_b(t, h(t)),$$ and has
 a unique zero in $[0, h(t)]$ which is nondegenedete.

We define
\begin{align*}
&T_1^{b}:=\sup\{s: w_b(t,0)>0\ \ {\rm for}\  t\in [1-\epsilon_0, s)\}, \\
&T_2^{b}:=\sup\{s: h(t)<L(b)\ \ {\rm for}\ t\in[1-\epsilon_0, s)\}.
\end{align*}
Clearly $T_1^{b}, T_2^{b}\ge 1+\epsilon_0$. Due to $h(t)\to +\infty$ and $w_b(t,0)\to -b<0$ as $t\to\infty$, both $T_1^{b}$ and $T_2^{b}$ are finite.

Now we are ready to present the estimate of $\theta(t)$ in the following lemma.

\begin{lemma}\label{thetatupper}
Let $(u,g,h)$ be the solution of \eqref{problem} with the initial datum $\tilde{u}_0$. Then there exists $M>0$ such that
\begin{align*}
\theta(t)\le Mt^{\frac{1}{2}-\frac{1}{p+1}}.
\end{align*}
\end{lemma}
\begin{proof}
Let $\sigma_1\in (0, \sigma)$ be arbitrarily given. By Lemma \ref{lb}, for any $h\in(L(\sigma_1), +\infty)$, there exists a unique $b\in (0, \sigma_1)$ such that $L(b)=h$.
By the fact $h(t)\to \infty$ as $t\to\infty$, there exists $T_0>0$ such that
\begin{align*}
h(t)>L(\sigma_1)\ \ \ {\rm for}\ t\ge T_0.
\end{align*}
Hence for any $t\ge T_0$, there exists a unique $b(t)\in (0, \sigma_1)$ such that
\begin{align*}
h(t)=L(b(t))\ \ \ {\rm for}\ t\ge T_0.
\end{align*}

{\bf Claim 1.} $\theta(t)\le l(b(t))$ for all large $t$.

Since $\lim_{t\to\infty}u(t,x)=\theta$ locally uniformly in $\mathbb{R}^1$, there exists $T_0'>T_0$ such that
\begin{align}\label{claim1-1}
u(t,0)<\theta+\sigma_1\ \ \ {\rm for}\ t\ge T_0'.
\end{align}
It is sufficient to prove that
 $$
 \mbox{$\theta(t)\le l(b(t))$ \  \ for\ $t\ge T_0'$.}
 $$

We proceed indirectly and suppose that there exists $t_0>T_0'$ such that $\theta(t_0)>l(b(t_0))=:l(b_0)$.
In view of $h(t_0)=L(b_0)$ and Lemmas 4.6--4.8 of \cite{DLZ}, we know that
\begin{align*}
u(t_0, x)-V_{b_0}(x)<0\ \ {\rm for}\ x<L(b_0)\  {\rm but\, close\,to}\ L(b_0).
\end{align*}
Recalling that $\theta(t_0)>l(b_0)$, it is easily seen that
\begin{align}\label{claim1-2}
u(t_0, l(b_0))-V_{b_0}(l(b_0))>0
\end{align}
Hence $u(t_0, x)-V_{b_0}(x)$ has a zero in $(l(b_0), L(b_0))$. This case can only happen in (ii) of Lemma 4.7 in \cite{DLZ}.
Then $u(t_0, x)-V_{b_0}(x)$ has the sign-changing pattern of $[+0-0]$. This and \eqref{claim1-2} imply that
\begin{align*}
&u(t_0, x)-V_{b_0}(x)>0\ \ \ {\rm for}\  0\le x<l(b_0),\\
&u(t_0, 0)-b_0-\theta>0.
\end{align*}

Define
\begin{align*}
&A:=\{b>b_0 : u(t_0,x)\ge V_{b}(x)\ge \theta\ \ {\rm in}\ [0, l(b)]\}, \\
&b_*:=\sup A.
\end{align*}
First we show that set $A$ is nonempty. We already have
\begin{align*}
u(t_0, x)>V_{b_0}(x)\ge\theta\ \ \ {\rm for}\  x\in[0, l(b_0)].
\end{align*}
Since $l(b)$ is nonincreasing in $b$ (due to Lemma \ref{lb}) and $V_{b}$ is continuous with respect to $b$, we obtain
\begin{align*}
u(t_0, x)>V_{b_2}(x)\ge\theta\ \ \ {\rm for}\ x\in [0, l(b_2)]
\end{align*}
for $b_2>b_0$ but close to $b_0$. This implies that $A$ is nonempty. In addition, from \eqref{claim1-1} for all $b\in A$, it follows
\begin{align*}
u(t_0, 0)\ge b+\theta\ \ {\rm that\ is,}\ \ b\le u(t_0, 0)-\theta\le\sigma_1.
\end{align*}
Then $b_*\le u(t_0, 0)-\theta\le\sigma_1$. Therefore $b_*$ is well-defined. In the sequel, we further claim that
 $$b_*=u(t_0, 0)-\theta.$$

Suppose that $b_*<u(t_0, 0)-\theta$. Then $V_{b_{*}}(0)<u(t_{0},0)$, that is, $w_{b_{*}}(t_{0}, 0)>0$. Note that $b^{*}>b_{0}$ and $L(b)$ is decreasing in $b$. So
 $$L(b_*)<L(b_0)=L(b(t_{0}))=h(t_{0}).$$
Then Lemmas 4.6--4.8 of \cite{DLZ} imply that $$T_2^{b_*}<t_0<T_1^{b_*}.$$

By Lemma 4.7 (iii) in \cite{DLZ}, $u(t_0, x)-V_{b_*}(x)$ has the sign-changing pattern $[+0-0+]$. We then conclude that
 $$
 u(t_0,x)>V_{b_*}(x)\ \ \ \mbox{for}\  x\in [0, l(b_*)].
 $$
For any $b\in(b_0, b_*)$, it follows that
\begin{align*}
u(t_0, x)\ge V_{b}(x)\ \ {\rm for\ all}\ x\in [0, l(b_*)]\subset [0, l(b)].
\end{align*}
Letting $b\nearrow b_*$ gives
\begin{align*}
u(t_0, x)\ge V_{b_*}(x)\ \ {\rm for}\ x\in [0, l(b_*)].
\end{align*}
Noting that
 $$
 V_{b_*}(l(b_*))=V_{b_0}(l(b_0))=\theta\ \ \mbox{and}\ \ u_x(t,x)<0,
 $$
we apply \eqref{claim1-2} and Lemma \ref{lb} to assert that
$$u(t,l(b_*))-V_{b_*}(l(b_*))\geq u(t,l(b_0))-V_{b_0}(l(b_0))>0.$$
Since $u(t_0, x)-V_{b_*}(x)$ has the sign-changing pattern $[+0-0+]$, there holds
\begin{align*}
u(t_0, x)>V_{b_*}(x)\ \ {\rm for}\ x\in [0, l(b_*)].
\end{align*}
By continuity of $V_b$ with respect to $b$, one can find $\varepsilon>0$ such that
\begin{align*}
u(t_0, x)>V_{b_*+\varepsilon}(x)\ \ {\rm for}\ x\in [0, l(b_*+\varepsilon)].
\end{align*}
This contradicts the definition of $b_*$. So we have $b_*=u(t_0, 0)-\theta$, that is, $u(t_0, 0)=V_{b_*}(0)$ and
\begin{align*}
u(t_0, x)\ge V_{b_*}(x)\ \ {\rm for}\ x\in [0, l(b_*)].
\end{align*}
However, by Lemmas 4.6--4.8 of \cite{DLZ}, if $u(t_0, x)-V_{b_*}(x)$ has a tangency at $x=0$, then $u(t_0,x)<V_{b_*}(x)$ in $(0,\delta)$ for some small $\delta>0$.
This is a contradiction. Thus we have shown that $\theta(t)\le l(b(t))$ for all large $t$.

{\bf Claim 2.} $\theta(t)\le Mt^{\frac{1}{2}-\frac{1}{p+1}}$ for all large $t$.

In the case of $p=1$, we have $l(b)=\frac{\pi}{2}$ by \eqref{lbest1}.

Now we handle the situation of $p>1$. By Theorem 1.2 in \cite{DLZ}(see Theorem \ref{DLZthm12}) and \eqref{Lb3}
\begin{align}
&h(t)=2\xi_0\sqrt{t}(1+o(1)), \label{Lbest1}\\
&h(t)=L(b(t))=C_{p}b(t)^{-\frac{p-1}{2}}+\sqrt{\frac{p+1}{2}}\theta b(t)^{-\frac{p+1}{2}}. \label{Lbest2}
\end{align}
By Lemma \ref{DLZLa43},
\begin{align*}
b\to 0\ \ {\rm if\ and\ only\ if}\ L(b)\to\infty.
\end{align*}
Since $h(t)\to\infty$ as $t\to\infty$, one further gets
\begin{align*}
b(t)\to 0,\ \ {\rm as}\ t\to\infty.
\end{align*}
From \eqref{Lbest1} and \eqref{Lbest2}, we have
\begin{align*}
2\xi_0\sqrt{t}(1+o(1))=b(t)^{-\frac{p+1}{2}}\sqrt{\frac{p+1}{2}}\theta(1+o(1)),\ \ {\rm as}\ t\to\infty
\end{align*}
and in turn
\begin{align*}
\frac{1}{\{t^{\frac{1}{p+1}}b(t)\}^{\frac{p+1}{2}}}=\frac{2\xi_0}{\theta}\sqrt{\frac{2}{p+1}}(1+o(1)),\ \ {\rm as}\ t\to\infty
\end{align*}
from which we obtain
\begin{align*}
b(t)=O(t^{-\frac{1}{p+1}}),\ \ {\rm as}\ t\to\infty.
\end{align*}
This and \eqref{lbest1} yield
\begin{align*}
&l(b(t))=O(t^{\frac{1}{2}-\frac{1}{p+1}}),\ \ {\rm as}\ t\to\infty.
\end{align*}
Therefore there exists $M>0$ such that
\begin{align*}
\theta(t)\le l(b(t))\le Mt^{\frac{1}{2}-\frac{1}{p+1}}\ \ \mbox{for all large}\ t.
\end{align*}

\end{proof}
Consider the Stefan problem
\begin{align}\label{Stefan0}
\left\{
\begin{array}{ll}
\rho_t-\rho_{xx}=0,                 &t>0, 0<x<r(t), \\
\rho(t,0)=\theta,\ \rho(t, r(t))=0, &t>0, \\
r'(t)=-\mu\rho_x(t, r(t)), &t>0.
\end{array}\right.
\end{align}
It is easily seen that the solution to \eqref{Stefan0} is given by
\begin{align}\label{SolofStefan}
\rho(t, x)=\theta\left[1-\frac{E\left(\frac{x}{2\sqrt{t}}\right)}{E(\xi_0)}\right],\  \
r(t)=2\xi_0\sqrt{t},
\end{align}
where
 $$E(x)=\frac{2}{\sqrt{\pi}}\int_0^x e^{-s^2}ds.$$

By utilizing the explicit solution of the Stefan problem \eqref{Stefan0}
to construct a suitable upper solution of \eqref{problem}, we can obtain the upper estimate of $h(t)$. Specifically, we can state
\begin{lemma}\label{l3.4}
If $p>1$ and let $(u,g,h)$ be the solution of \eqref{problem} with the initial datum $\tilde{u}_0$, we have
\begin{align*}
h(t)\le 2\xi_0\sqrt{t}+M t^{\frac{1}{2}-\frac{1}{p+1}},\ \ \ {\rm as}\ t\to\infty.
\end{align*}
\end{lemma}
\begin{proof}
In view of Lemma \ref{thetatupper}, we can find some $K>0$ and $T_0>1$ such that
\begin{align*}
\theta(t)\le Kt^{\frac{1}{2}-\frac{1}{p+1}}\ \ {\rm for}\ t\ge T_0.
\end{align*}
Fix $T_0$ and choose $M>0$ so that
\begin{align*}
M>\max\{K,\ \ T_0^{\frac{1}{p+1}-\frac{1}{2}}h(T_0)\}.
\end{align*}
Define
\begin{align*}
&\xi(t)=Mt^{\frac{1}{2}-\frac{1}{p+1}},\\
&\overline{u}(t,x)=\rho(t,x-\xi(t)),\\
&\overline{h}(t)=\xi(t)+r(t),
\end{align*}
where $\rho$, $r$ is given by \eqref{SolofStefan}.

We will check that $(\overline{u}, \xi, \overline{h})$ satisfies the condition of Lemma \ref{comp2}, that is,
\begin{align}
&\overline{h}(T_0)\ge \xi(T_0)>h(T_0), \label{6} \\
&\left\{\begin{array}{ll}
\overline{u}(T_0, x)\ge 0\ \ &{\rm for}\ x\in[\xi(T_0), \overline{h}(T_0)], \\
\overline{u}(T_0, x)>0\ \ &{\rm for}\ x\in(\xi(T_0), \overline{h}(T_0)),
\end{array}\right. \label{7} \\
&\left\{\begin{array}{ll}
\overline{u}(t, \xi(t))\ge 0\ &{\rm if}\ \xi(t)>h(t), \\
\overline{u}(t, \xi(t))\ge u(t, \xi(t))\ &{\rm if}\ \xi(t)\le h(t),
\end{array}\right. {\rm for}\ t\ge T_0,\label{8} \\
&\overline{u}(t, \overline{h}(t))=0\ \ {\rm for}\ \ t>T_0, \label{9} \\
&\overline{h}'(t)\ge -\mu \overline{u}_x(t, \overline{h}(t))\ \ {\rm for}\ \ t>T_0, \label{10} \\
&\overline{u}_t-\overline{u}_{xx}\ge f(\overline{u})\ \ {\rm for}\ \ t>T_0,\ x\in(\xi(t), h(t)). \label{11}
\end{align}
From the choice of $M$, it follows that
\begin{align*}
\overline{h}(T_0)=\xi(T_0)+r(T_0)\ge\xi(T_0)=MT_0^{\frac{1}{2}-\frac{1}{p+1}}\ge h(T_0),
\end{align*}
and so \eqref{6} holds.

By the definition of $\overline{u}$, \eqref{7} is trivially true.

When $t\ge T_0$ and $\xi(t)>h(t)$, clearly $\overline{u}(t, \xi(t))\ge 0$. As
 $$\theta(t)\le Kt^{\frac{1}{2}-\frac{1}{p+1}}< \xi(t)\ \ \mbox{for}\ t\ge T_0,$$  we get
\begin{align*}
\overline{u}(t,\xi(t))=\theta=u(t,\theta(t))\ge u(t,\xi(t))
\end{align*}
for $t\ge T_0$, $ \xi(t)<h(t)$. Thus \eqref{8} holds.

By the definition of $\overline{u}$ again, we have $\overline{u}(t,\overline{h}(t))=\rho(t, r(t))=0$, verifying \eqref{9}.

We now show that \eqref{10} holds. In fact, for $t\ge T_0$,
\begin{align*}
\overline{h}'(t)=r'(t)+\xi'(t)&=-\mu\rho_x(t,r(t))+\frac{M(p-1)}{2(p+1)}t^{-\frac{p+3}{2(p+1)}} \\
                              &>-\mu\rho_x(t,r(t))=-\mu\overline{u}_x(t, \overline{h}(t)).
\end{align*}
Therefore \eqref{10} is fulfilled.

Finally we verify \eqref{11}. Direct calculation gives
\begin{align*}
&\overline{u}_t=\rho_t(t, x-\xi(t))-\xi'(t)\rho_x(t, x-\xi(t)), \\
&\overline{u}_{xx}=\rho_{xx}(t, x-\xi(t)).
\end{align*}
Then we have, for $t\ge T_1$ and $\xi(t)<x<\overline{h}(t)$,
\begin{align*}
\overline{u}_t-\overline{u}_{xx}=-\xi'(t)\rho_x(t, x-\xi(t)).
\end{align*}
Since
\begin{align*}
\rho_x(t,x)=-\frac{\theta}{E(\xi_0)\sqrt{\pi t}}e^{-\frac{x^2}{4t}},
\end{align*}
we find
\begin{align*}
\overline{u}_t-\overline{u}_{xx}&=-\xi'(t)\rho_x(t, x-\xi(t)) \\
                                &=\frac{p-1}{2(p+1)}Mt^{\frac{p-1}{2(p+1)}-1}\frac{\theta}{E(\xi_0)\sqrt{\pi t}}e^{-\frac{(x-\xi(t))^2}{4t}}\ge 0=f(\overline{u}).
\end{align*}
Here we used the facts that $0\le \overline{u}\le\theta$ and $f(\overline{u})=0$. Consequently, \eqref{11} is satisfied.

With the help of Lemma \ref{comp2}, we can assert
\begin{align*}
h(t)&\le 2\xi_0\sqrt{t}+Mt^{\frac{1}{2}-\frac{1}{p+1}},\ \ \mbox{for}\ t\ge T_0.
\end{align*}
The proof is thus complete.
\end{proof}

\begin{lemma}
Assume that $p=1$ and let $(u,g,h)$ be the solution of \eqref{problem} with the initial datum $\tilde{u}_0$. Then we have
$$h(t)\leq 2\xi_0\sqrt{t}+O(1).$$
\end{lemma}
\begin{proof}
When $p=1$, it is clear that $l(b)=\pi/2$ for $b\in(0,\delta)$. On the other hand, from the proof of Claim 1 in the lemma above,
we have $\theta(t)\le l(b(t))$ for large $t$. By choosing a sufficiently large constant $M$, one can easily check that
$(\bar{u},\bar{h}):=(\rho(t,x-M),r(t)+M)$ is an upper solution of $(u,g,h)$ by a similar argument as in Lemma \ref{l3.4}.
Therefore, $h(t)\leq 2\xi_0\sqrt{t}+M$ follows.
\end{proof}
\subsection{Lower bound estimate for $p>3$}

Our proof of the upper bound of $h(t)$ heavily relies on the upper bound estimate for $\theta(t)$;
it seems that such an idea ceases to produce a refined lower bound of $h(t)$. To the aim, in the following we will develop a different approach by firstly deriving a lower bound of $u(t,0)$.

\begin{lemma}\label{lemma1}
Assume that $p>3$ and let $(u,g,h)$ be the solution of \eqref{problem} with the initial datum $\tilde{u}_0$. Then there exist constants $a>0$ and $T>0$ such that
\begin{align*}
u(t, 0)\ge \theta+a(t+T)^{-\frac{1}{p-1}},\ \ \forall t\geq0.
\end{align*}

\end{lemma}
\begin{proof}
We first recall a well known result of the self-similar solution to the following nonlinear heat equation
\begin{align}\label{NLH}
v_t=v_{xx}+|v|^{p-1}v.
\end{align}
Consider an associated problem
\begin{align*}
\left\{
\begin{array}{ll}
\displaystyle\frac{\varphi}{p-1}+\frac{y}{2}\varphi'+\varphi''+|\varphi|^{p-1}\varphi=0,\ \ y>0,\\
\varphi(0)=\gamma,\ \varphi'(0)=0.
\end{array}\right.
\end{align*}
By the result of Haraux and Weissler \cite{HW}, for
 $$
 p>3\ \ \mbox{and}\ \ 0<\gamma<\left\{[p-3]/2(p-1)\right\}^{\frac{1}{p-1}},
 $$
the above problem has a unique solution $\varphi$, which is defined for all $y>0$, and
$\varphi(y)>0, \ \forall y\ge 0$. If we define
 $$v(t,x)=t^{-1/(p-1)}\varphi(x/\sqrt{t}),$$ then $v$ satisfies \eqref{NLH}.

Set
  $$
  w(t,x)=\theta+v(t,x).
  $$
Then $w$ solves
\begin{align*}
w_t-w_{xx}=|w-\theta|^{p-1}(w-\theta),\ \ t>0,\ x\in\mathbb{R}.
\end{align*}
In view of $v(t,\cdot)\to 0$ as $t\to\infty$ locally uniformly in $\mathbb{R}$ (actually in $C^1(\mathbb{R})$),
it follows that $w(t,\cdot)\to\theta$ locally uniformly in $\mathbb{R}$ as $t\to\infty$.

Recall that $u(t,0)>\theta$ for all $t\ge 0$.  Thus there exists $T>0$ such that
\begin{align*}
u(0,0)>w(T_{0}, 0).
\end{align*}
Furthermore we may assume that $u(0,x)$ and $w(T, x)$ have exact one intersection point $y\in (0, h_0)$.

Denote
 $$\eta(t,x)=u(t,x)-w(t+T, x)\ \ \mbox{for}\ t\ge 0,\ x\in[0,h(t)].$$ Then we claim
\begin{align*}
\eta(t,0)>0\ \ {\rm for\ all}\ t>0.
\end{align*}

Otherwise there exists $t_1>0$ such that
\begin{align*}
\eta(t, 0)>0\ \ {\rm for}\ t\in[0, t_1)\  {\rm and}\  \eta(t_1, 0)=0.
\end{align*}
Since $\eta(t,0)>0>\eta(t, h(t))$ for $t\in (0,t_1)$ and the unique existence of $y(t)$ (zero of $\eta$) for $t\in(0, \varepsilon_1)$ ($\varepsilon_1>0$ is a small constant), we deduce that
$\eta(t,x)$ has a unique nondegenerate zero over $[0, h(t)]$ for $t\in (0, t_1)$ due to Lemma 2.2 in \cite{DLZ}. Hence $y(t)$ can be extended to all $t\in(0,t_1)$.

Let us look at the limit of $y(t)$ as $t$ increases to $t_1$. If the limit does not exist, then as the proof of Lemma 2.4 in \cite{DLZ} we have $\eta(t_1, x)\equiv 0$ over $[0, h(t_1)]$ which
contradicts with $\eta(t_1, h(t_1))<0$. Hence the limit exists and we denote it by $y(t_1)$. As $\eta(t_1, y(t_1))=0$, clearly $y(t_1)<h(t_1)$. If $y(t_1)=0$, the maximum principle,
as applied to $\eta$ over $\{(t,x)|0<t<t_1, y(t)<x<h(t)\}$, gives
 $$\eta(t_1, x)<0\ \ \mbox{for}\ x\in (0, h(t_1)].$$ It is easily seen that
 $$u(t_1, x)<w(t_1+T, x)\ \ \mbox{for}\ x\in [-h(t_1), h(t_1)]\backslash \{0\}.$$
By the strong maximum principle,
 \begin{align*}
 u(t,x)<w(t+T, x)\ \ {\rm for}\  t>t_1,\  x\in [-h(t), h(t)].
 \end{align*}
For any $t_2>t_1$, we choose $\varepsilon_0>0$ small enough such that
\begin{align*}
(1+\varepsilon_0)u(t_2, x)\le w(t_2+T, x)\ \ {\rm for}\ 0\le x\le h(t_2).
\end{align*}
We can apply the comparison principle to conclude that
\begin{align*}
u(t, x; (1+\varepsilon_0)u(t_2, x))\le w(t+t_2+T, x)\ \ \ {\rm for}\  t\ge 0,\  0\le x\le h(t+t_2).
\end{align*}
By taking $(1+\varepsilon_0)u(t_2, \cdot)$ as a new initial value, Theorem \ref{theoremDL} concludes that
\begin{align*}
u(t, x; (1+\varepsilon_0)u(t_2, x))\to 1,\ \ {\rm as}\ t\to\infty,
\end{align*}
which is a contradiction.

If $y(t_1)\in (0, h(t_1))$, then by the maximum principle applied to $\eta$ over $\{(t, x)|0\le t\le t_1,\ 0<x<y(t)\}$,
one can see that $\eta(t_1, x)>0$ for $0<x<y(t_1)$. In light of the Hopf lemma, we further have
\begin{align*}
\eta_x(t_1,0)>0,
\end{align*}
which implies $$u_x(t_1, 0)>w_x(t_1+T, 0)=v_x(t_1+T, 0)=0.$$ This is a contradiction to $u_x(t_1, 0)=0$. Thus, our previous claim holds.

As  a result, we derive
\begin{align*}
u(t, 0)>w(t+T, 0)=\theta+(t+T)^{-\frac{1}{p-1}}\varphi(0)=\theta+\gamma (t+T)^{-\frac{1}{p-1}}.
\end{align*}
Denoting $a=\gamma$, we obtain the desired assertion.
\end{proof}

Inspired by the Stefan problem \eqref{SolofStefan}, we are now able to deduce the lower bound of $h(t)$ using Lemma \ref{lemma1}. Indeed, we have
\begin{lemma}
Assume that $p>3$ and Let $(u,g,h)$ be the solution of \eqref{problem} with the initial datum $\tilde{u}_0$. Then there exists a constant $m>0$ such that
\begin{align*}
h(t)\ge 2\xi_0\sqrt{t}+mt^{\frac{1}{2}-\frac{1}{p-1}},\ \ {\rm as}\  t\to\infty.
\end{align*}
where $\xi_0>0$ is given by \eqref{xi0}.
\end{lemma}
\begin{proof}
Let
 $$\xi(t)=b(t+T)^{-\frac{1}{p-1}},$$ where $T$ is the constant determined by Lemma \ref{lemma1} and $0<b<a$ is a constant to be chosen later.
We first observe that there exists $\beta\in C^1$ such that $\beta(t)>0$ for $t>0$ and
\begin{align}\label{betat}
\sqrt{\pi}\beta(t)e^{\beta(t)^2}E(\beta(t))=\mu(\theta+\xi(t)),
\end{align}
where $E$ is defined in \eqref{SolofStefan}. In fact, by setting $$\Phi(x)=\sqrt{\pi}xe^{x^2}E(x),$$ we have
\begin{align*}
\Phi'(x)=\sqrt{\pi}e^{x^2}E(x)+2\sqrt{\pi}x^2e^{x^2}E(x)+2x>0\ \ {\rm for}\ x>0.
\end{align*}
Since $\Phi(0)=0$ and $\lim_{x\to\infty}\Phi(x)=\infty$, there exists a unique $\beta=\beta(t)$ such that $\Phi(\beta(t))=\mu(\theta+\xi(t))$.
The implicit function theorem guarantees that $\beta\in C^1$. One can further claim that $\beta'(t)<0$. Indeed, by \eqref{betat},
\begin{align*}
\beta'(t)\Phi'(\beta(t))=\mu\xi'(t)=-\frac{b\mu}{p-1}(t+T)^{-\frac{1}{p-1}-1}<0\ \ {\rm for}\ t>0,
\end{align*}
which indicates $\beta'(t)<0$ for $t>0$.

Now we define a lower solution as follows:
\begin{align*}
\underline{h}(t)=2\beta(t)\sqrt{t},\ \ \underline{u}(t,x)=(\theta+\xi(t))\left\{1-\frac{E\left(\frac{x}{2\sqrt{t}}\right)}{E(\beta(t))}\right\}.
\end{align*}
Fix $T_0>0$ and choose $b$ so that
\begin{align}
&b<a\left(\frac{T}{T+T_0}\right)^{\frac{1}{p-1}},\ \ {\rm that\ is},\ \ bT^{-\frac{1}{p-1}}<a(T+T_0)^{-\frac{1}{p-1}},\label{ba} \\
&b<\sigma T^{\frac{1}{p-1}}. \nonumber
\end{align}
Then we have
 $$\underline{u}(t,x)\le\theta+\sigma,\ \ \ f(\underline{u})\ge 0.$$

We next show that for sufficiently small $\varepsilon>0$, $(\underline{u}, \xi, \underline{h})$ is
a lower solution to \eqref{problem}, that is,
\begin{align}
&\underline{u}_t-\underline{u}_{xx}-f(\underline{u})\le 0,\ \ t>\varepsilon,\ 0<x<\underline{h}(t), \label{1} \\
&\underline{u}(t, \underline{h}(t))=0,\ \ t>\varepsilon, \label{2}\\
&\underline{h}(\varepsilon)\le h(T_{0}),\ \underline{u}(\varepsilon, x)\le u(T_0, x),\  0\le x\le \underline{h}(\varepsilon),\label{3}\\
&\underline{u}(t, 0)\le u(t-\varepsilon+T_0, 0),\ \ t>\varepsilon, \label{4}\\
&\underline{h}'(t)\le -\mu \underline{u}_x(t, \underline{h}(t)),\ \ t>\varepsilon. \label{5}
\end{align}

Basic calculation yields
\begin{align*}
 &\underline{u}_t-\underline{u}_{xx}-f(\underline{u}) \\\
=&\xi'(t)\left\{1-\frac{E\left(\frac{x}{2\sqrt{t}}\right)}{E(\beta(t))}\right\}-(\theta+\xi(t))\left\{\frac{E\left(\frac{x}{2\sqrt{t}}\right)}{E(\beta(t))}\right\}_t+(\theta+\xi(t))\left\{\frac{E\left(\frac{x}{2\sqrt{t}}\right)}{E(\beta(t))}\right\}_{xx}-f(\underline{u}) \\
=&-\frac{b}{p-1}(t+T)^{-\frac{1}{p-1}-1}\left\{1-\frac{E\left(\frac{x}{2\sqrt{t}}\right)}{E(\beta(t))}\right\} \\
 &\ \ \ -(\theta+\xi(t))\frac{-\frac{1}{2\sqrt{\pi}}xt^{-\frac{3}{2}}e^{-\frac{x^2}{4t}}E(\beta(t))-\frac{2}{\sqrt{\pi}}\beta'(t)e^{-\beta(t)^2}E\left(\frac{x}{2\sqrt{t}}\right)}{[E(\beta(t))]^2} \\
 &\ \ \ -(\theta+\xi(t))\frac{x}{2\sqrt{\pi}t^{\frac{3}{2}}}\cdot\frac{1}{E(\beta(t))}e^{-\frac{x^2}{4t}}-f(\underline{u}) \\
=&-\frac{b}{p-1}(t+T)^{-\frac{1}{p-1}-1}\left\{1-\frac{E\left(\frac{x}{2\sqrt{t}}\right)}{E(\beta(t))}\right\}+(\theta+\xi(t))\frac{2\beta'(t)e^{-\beta(t)^2}E\left(\frac{x}{2\sqrt{t}}\right)}{\sqrt{\pi}[E(\beta(t))]^2}-f(\underline{u}).
\end{align*}
As $\beta'(t)<0$, $\underline{u}_t-\underline{u}_{xx}-f(\underline{u})\le 0$, and so \eqref{1} holds.

By the definition, clearly \eqref{2} is true.

We then check \eqref{3}. By Lemma \ref{lemma1} and \eqref{ba} we have
\begin{align*}
\underline{u}(\varepsilon, x)\le \underline{u}(\varepsilon, 0)\le \theta+b(\varepsilon+T)^{-\frac{1}{p-1}}&<\theta+bT^{-\frac{1}{p-1}}=\underline{u}(0,0) \\
                      &<\theta+a(T+T_0)^{-\frac{1}{p-1}}\le u(T_0, 0)
\end{align*}
for any sufficiently small $\varepsilon>0$ and $0\le x\le \underline{h}(\varepsilon)$. By the continuity of $u(T_0, x)$, one can find a small $\delta>0$ such that
\begin{align*}
\underline{u}(0,0)=\theta+bT^{-\frac{1}{p-1}}<u(T_0, x)\ \ {\rm for}\ x\in[0, \delta).
\end{align*}
By choosing $\varepsilon>0$ to be small enough so that
\begin{align*}
\varepsilon<\frac{1}{4\beta(0)^{2}}\min\{\delta^{2}, h_{0}^{2}\},\ \ {\rm that\ is},\ \ 2\beta(\varepsilon)\sqrt{\varepsilon}<2\beta(0)\sqrt{\varepsilon}<\min\{\delta, h_{0}\},
\end{align*}
we deduce
\begin{align*}
\underline{h}(\varepsilon)=2\beta(\varepsilon)\sqrt{\varepsilon}<h_{0}<h(T_{0})
\end{align*}
and
\begin{align*}
\underline{u}(\varepsilon, x)<\theta+bT^{-\frac{1}{p-1}}\le u(T_{0}, x)
\end{align*}
for $x\in [0, \underline{h}(\varepsilon)]=[0, 2\beta(\varepsilon)\sqrt{\varepsilon}]$. This indicates \eqref{3}.

In view of \eqref{ba}, we have
\begin{align*}
\underline{u}(t,0)&=\theta+b(t+T)^{-\frac{1}{p-1}} \\
                  &\le\theta+b(t-\varepsilon+T)^{-\frac{1}{p-1}} \\
                  &<\theta+a(t-\varepsilon+T+T_0)^{-\frac{1}{p-1}}\le u(t-\varepsilon+T_0, 0),
\end{align*}
which verifies \eqref{4}.

Finally we check \eqref{5}. From the definition of $\beta(t)$ and \eqref{betat}, it follows
\begin{align*}
-\mu\underline{u}_x(t, \underline{h}(t))=\mu(\theta+\xi(t))\frac{\frac{2}{\sqrt{\pi}}\cdot\frac{1}{2\sqrt{t}}e^{-\beta(t)^2}}{E(\beta(t))}=\beta(t)\frac{1}{\sqrt{t}}.
\end{align*}
Therefore,
\begin{align*}
\underline{h}'(t)=2\beta'(t)\sqrt{t}+\beta(t)\frac{1}{\sqrt{t}}\le\beta(t)\frac{1}{\sqrt{t}}=-\mu\underline{u}_x(t, \underline{h}(t)).
\end{align*}

Now, applying the comparison principle (Lemma \ref{comp1}),  we can conclude
\begin{align*}
h(t-\varepsilon+T_0)\ge \underline{h}(t)\ \ {\rm for}\ \ t\ge\varepsilon.
\end{align*}
In view of
\begin{align*}
&\Phi(\beta(t))=\Phi(\xi_0)+\mu b(t+T)^{-\frac{1}{p-1}}, \\
&\Phi'(\zeta)(\beta(t)-\xi_0)=\mu b(t+T)^{-\frac{1}{p-1}}\ \ (\xi_0\le\zeta\le\beta(t)),
\end{align*}
we have
\begin{align*}
\beta(t)-\xi_0=\frac{\mu b}{\Phi'(\zeta)}(t+T)^{-\frac{1}{p-1}}.
\end{align*}
As a consequence, we obtain
\begin{align*}
h(t)&\ge\underline{h}(t+\varepsilon-T_0) \\
    &=2\left(\xi_0+\frac{\mu b}{\Phi'(\zeta)}(t+\varepsilon-T_0+T)^{-\frac{1}{p-1}}\right)(t+\varepsilon-T_0)^{\frac{1}{2}}\\
    &\ge 2\left(\xi_0+\frac{\mu b}{2\Phi'(\zeta)}(t+\varepsilon-T_0)^{-\frac{1}{p-1}}\right)(t+\varepsilon-T_0)^{\frac{1}{2}}
\end{align*}
for sufficiently large $t$, and then there exists $m>0$ such that
\begin{align*}
h(t)&\ge 2\xi_0(t+\varepsilon-T_0)^{\frac{1}{2}}+m(t+\varepsilon-T_0)^{\frac{1}{2}-\frac{1}{p-1}} \\
    &\ge 2\xi_0\sqrt{t}+mt^{\frac{1}{2}-\frac{1}{p-1}}+o(1),\ \ \mbox{as}\ t\to\infty.
\end{align*}

\end{proof}

\subsection{Lower bound estimate for $1\le p\le 3$}

It turns out that the lower bound estimate of $h(t)$ in the case $1\le p\le 3$ is easier to establish.

\begin{lemma}\label{l3.8}
Assume that $1\leq p\leq 3$ and let $(u,g,h)$ be the solution of \eqref{problem} with the initial datum $\tilde{u}_0$. Then we have
\begin{align*}
h(t)\ge 2\xi_0\sqrt{t}+O(1),\ \ {\rm as}\  t\to\infty,
\end{align*}
where $\xi_0>0$ is given by \eqref{xi0}.
\end{lemma}
\begin{proof}

We construct a lower solution in the following manner:
\begin{align*}
\underline{u}(t, x)=\theta\left[1-\frac{E\left(\frac{x}{2\sqrt{t}}\right)}{E(\xi_0)}\right],\ \ \underline{h}(t)=2\xi_0\sqrt{t}.
\end{align*}
Then we have $\underline{u}(t, x)\le \theta$ and $f(\underline{u})\ge 0$.

We will show that for sufficiently small $\varepsilon>0$ $(\underline{u}, \xi, \underline{h})$ is a lower solution to \eqref{problem} for $t>\varepsilon$, that is,
\begin{align}
&\underline{u}_t-\underline{u}_{xx}-f(\underline{u})\le 0,\ \ t>\varepsilon,\ 0<x<\underline{h}(t), \label{01} \\
&\underline{u}(t, \underline{h}(t))=0,\ \ t>\varepsilon, \label{02}\\
&\underline{h}(\varepsilon)\le h_0,\ \underline{u}(\varepsilon, x)\le u(0, x),\  0\le x\le \underline{h}(\varepsilon),\label{03}\\
&\underline{u}(t, 0)\le u(t-\varepsilon, 0),\ \ t>\varepsilon, \label{04}\\
&\underline{h}'(t)\le -\mu \underline{u}_x(t, \underline{h}(t)),\ \ t>\varepsilon. \label{05}
\end{align}
From \eqref{Stefan0} and the fact that $f(\underline{u})\ge 0$, \eqref{01} follows.

By the definition, it is easily seen that \eqref{02} holds.

Next we check \eqref{03}. As $u(t, 0)>\theta$ for $t\ge 0$, clearly
\begin{align*}
\underline{u}(\varepsilon, x)\le \underline{u}(\varepsilon, 0)\le \theta<u(0, 0)
\end{align*}
for any sufficiently small $\varepsilon>0$ and $0\le x\le \underline{h}(\varepsilon)$.

By the continuity of $u(0, x)$, we can find a small $\delta>0$ such that
\begin{align*}
u(0, x)>\theta\ \ {\rm for}\ x\in[0, \delta).
\end{align*}
Taking $\varepsilon>0$ to be sufficiently small so that
\begin{align*}
\varepsilon<\frac{1}{4\xi_0^2}\min\{\delta^{2}, h_{0}^{2}\},\ \ {\rm that\ is},\ \ 2\xi_0\sqrt{\varepsilon}<\min\{\delta, h_{0}\},
\end{align*}
we have
\begin{align*}
\underline{h}(\varepsilon)=2\xi_0\sqrt{\varepsilon}<h_{0}
\end{align*}
and
\begin{align*}
\underline{u}(\varepsilon, x)\le \theta<u(0, x)
\end{align*}
for $x\in [0, \underline{h}(\varepsilon)]=[0, 2\xi_0\sqrt{\varepsilon}]$.

In view of $\underline{u}(t,0)\le \theta<u(t-\varepsilon, 0)$ for $t\ge\varepsilon$, it is also clear that \eqref{04} holds.

Finally \eqref{Stefan0} ensures that \eqref{05} is satisfied.

Hence, Lemma \ref{comp1} yields that
\begin{align*}
\underline{h}(t+\varepsilon)=2\xi_0\sqrt{t+\varepsilon}\le h(t)\ \ {\rm for}\ t>0.
\end{align*}
and in turn,
\begin{align*}
h(t)\ge 2\xi_0\sqrt{t}+O(1),\ \ \mbox{as}\ t\to\infty.
\end{align*}

\end{proof}

Combining Lemmas \ref{l3.4}--\ref{l3.8}, we obtain Theorem \ref{theoremA}.

\end{document}